\RequirePackage[l2tabu, orthodox]{nag} % In­hibit use of non-ams­math math­e­mat­ics markup when us­ing ams­math 
\documentclass[final, a4paper, 11pt]{article} % For publication (checking an expected number of pages)

\usepackage[all, warning]{onlyamsmath} % De­tect­ing and warn­ing aboutt ob­so­lete LaTeX com­mands
\usepackage[T1]{fontenc}
\usepackage{fix-cm}
\usepackage{exscale, fullpage}
\usepackage{graphicx}
\usepackage{amssymb}
\usepackage{amsthm,amsmath}
\usepackage{hyperxmp}
\usepackage[
bookmarks=true, 
bookmarksnumbered=true, 
bookmarkstype=toc, 
colorlinks={true},%
pdfauthor={Akira Imakura, Keiichi Morikuni, Takayasu Akitoshi},%
pdfdisplaydoctitle={true},
pdfpagemode=UseNone,
pdfstartview=FitH,
pdftitle={Verified partial eigenvalue computations using contour integrals for Hermitian generalized eigenproblems},
setpagesize={false},%
urlcolor={red} %,
]{hyperref}
\usepackage{algorithm, algorithmic}
\usepackage[subrefformat=parens]{subcaption}
\usepackage{url}

\usepackage{cite}

\theoremstyle{plain}
\newtheorem{theorem}{Theorem}[section]
\newtheorem{lemma}[theorem]{Lemma}

\newcommand{\Ker}{\mathrm{Ker}}

\newcommand{\diag}{\mathrm{diag}}

\title{Verified partial eigenvalue computations using contour integrals\\for Hermitian generalized eigenproblems}
\author{Akira Imakura\thanks{Faculty of Engineering, Information and Systems, University of Tsukuba, 1-1-1 Tennodai, Tsukuba, Ibaraki 305-8573, Japan} \thanks{\texttt{imakura@cs.tsukuba.ac.jp}} \and Keiichi Morikuni\footnotemark[1] \thanks{\texttt{morikuni@cs.tsukuba.ac.jp}} \and Akitoshi Takayasu\footnotemark[1] \thanks{\texttt{takitoshi@risk.tsukuba.ac.jp}}}
\date{}

\begin{document}	
\maketitle
		
\begin{abstract}
	We propose a verified computation method for partial eigenvalues of a Hermitian generalized eigenproblem.
	The block Sakurai--Sugiura Hankel method, a contour integral-type eigensolver, can reduce a given eigenproblem into a generalized eigenproblem of block Hankel matrices whose entries consist of complex moments.
	In this study, we evaluate all errors in computing the complex moments.
	We derive a truncation error bound of the quadrature.
	Then, we take numerical errors of the quadrature into account and rigorously enclose the entries of the block Hankel matrices.
	Each quadrature point gives rise to a linear system, and its structure enables us to develop an efficient technique to verify the approximate solution.	
	Numerical experiments show that the proposed method outperforms a standard method and infer that the proposed method is potentially efficient in parallel.
\end{abstract}


		
%	\begin{keyword}
%	partial eigenproblem, contour integral, complex moment, verified numerical computations
			
%	\MSC 65F15 \sep 65G20 \sep 65G50
%	\end{keyword}		

\section{Introduction}
We consider verifying the $m$ eigenvalues $\lambda_i$, counting multiplicity, of the Hermitian generalized eigenproblem 
\begin{equation}
	A \boldsymbol{x}_i = \lambda_i B \boldsymbol{x}_i, \quad \boldsymbol{x}_i \in \mathbb{C}^n \setminus \{ \boldsymbol{0} \}, \quad i = 1, 2, \dots, m
	\label{eq:gevp}
\end{equation}
in a prescribed interval $\Omega = [a, b] \subset \mathbb{R}$, where $A = A^\mathsf{H} \in \mathbb{C}^{n \times n}$, $B = B^\mathsf{H} \in \mathbb{C}^{n \times n}$ is positive semidefinite, and the matrix pencil $zB-A$ ($z\in\mathbb{C}$) is regular\footnote{See Appends~\ref{app:regular_pencil} for the verification of regularity of a matrix pencil.}, i.e, $\det (zB-A)$ is not identically equal to zero.
We call $\lambda_i$ an \emph{eigenvalue} and $\boldsymbol{x}_i$ the corresponding \emph{eigenvector} of the problem \eqref{eq:gevp} or \emph{matrix pencil} $z B - A$, $z \in \mathbb{C}$ interchangeably.
Throughout, we assume that the number of eigenvalues in the interval $\Omega$ is known to be $m$ and there do not exist eigenvalues of \eqref{eq:gevp} at the end points $a$, $b \in \mathbb{R}$.
We also denote the eigenvalues of \eqref{eq:gevp} outside $\Omega$ by $\lambda_i$ ($i = m+1, m+2, \dots, r$), where $r = \mathrm{rank}\,B$.
\par
There are plenty of previous works for verification methods of eigenvalue problems (see, e.g., \cite{Rump2010} and references therein).
These previous works, in particular, for symmetric generalized eigenvalue problems are classified into two kinds: some of them aim at rigorously enclosing specific eigenvalues, and others aim at rigorously enclosing all eigenvalues.
For the purposes, different approaches have been taken.
Behnke~\cite{Behnke1988} used Temple quotients, their generalizations, and the LDLT decomposition to verify specific eigenvalues.
%Rump~\cite{Rump1989} regarded a non-Hermitian generalized eigenproblem as systems of nonlinear equations and used Newton-like iterations for solving the equations to verify specific eigenpairs of the eigenproblem, even for singular $B$. % (The inclusion of multiple eigenvalues is open.)
Behnke~\cite{Behnke1991} used the variational principle to verify specific eigenvalues.
Watanabe et al.~\cite{WatanabeYamamotoNakao1999TJSIAM} used an approximate diagonalization and generalized Rump's method, avoiding the Cholesky factorization, to verify the eigenvalue with the maximum magnitude.
Yamamoto~\cite{Yamamoto2001LAA} combined the LDLT decomposition with Sylvester's law of inertia to verify specific eigenvalues.
%Rump~\cite{Rump2001LAA} used Newton-like iterations for the non-Hermitian generalized eigenproblems with nonsingular $B$ to verify specific eigenvalues and a basis of the corresponding invariant subspaces, allowing multiple eigenvalues.
Maruyama et al.~\cite{MaruyamaOgitaNakayaOishi2004} used Ger{\v s}hgorin's theorem to verify all eigenpairs.
Miyajima et al.~\cite{MiyajimaOgitaRumpOishi2010} used the techniques in \cite{MiyajimaOgitaOishi2005TJSIAM, MiyajimaOgitaOishi2006TJSIAM} and combined it with Rump and Wilkinson's bounds to verify all eigenpairs.
%Miyajima~\cite{Miyajima2010JCAM} derived a variant of the Bauer--File theorem to determine a single error bound for enclosing all eigenvalues, and Miyajima~\cite{Miyajima2012JCAM} extended this idea to give an error bound for enclosing each eigenvalue of non-Hermitian generalized eigenproblems with nonsingular $B$, allowing multiple eigenvalues.
%Miyajima~\cite{Miyajima2014SIMAX} transformed the problem \eqref{eq:gevp} to nonlinear equations and uses the block diagonalization and Ger{\v s}hgorin's theorem for non-Hermitian generalized eigenproblems with nonsingular $B$, allowing defective eigenvalues.
%See \cite{Rump1989, Rump2001LAA, Miyajima2010JCAM, Miyajima2012JCAM, Miyajima2014SIMAX} for the non-Hermitian case.
See \cite{Miyajima2014SIMAX} and references therein for the non-Hermitian case.
%application: numerical verification of solutions of elliptic equations in nonconvex domains reduces to an eigenvalue problem \cite{YamamotoNakao1993}. constant in a priori error estimations in finite element
\par
In this study, we develop a verification method for partial eigenvalues using the block Hankel-type Sakurai--Sugiura method \cite{IkegamiSakuraiNagashima2010}, which receives attentions in recent years by virtue of the scalability in parallel and versatility \cite{ImakuraDuSakurai2016JJIAM}.
We shed light on a new perspective of this method.
This method uses contour integrals to form complex moment matrices.
Their truncation errors for the trapezoidal rule of numerical quadrature were derived by Miyata et al.~\cite{MiyataDuSogabeYamamotoZhang2009}.
Thanks to their work, we derive a numerically computable enclosure of the complex moment.
We point out that our verification method works for multiple eigenvalues in the prescribed region and for semidefinite $B$, whereas the previous methods \cite{Behnke1988, Behnke1991, Yamamoto2001LAA, MaruyamaOgitaNakayaOishi2004, MiyajimaOgitaOishi2005TJSIAM, MiyajimaOgitaOishi2006TJSIAM, MiyajimaOgitaRumpOishi2010} work only for positive definite $B$.
In addition, for each quadrature point, a structured linear system of equations arises to solve.
The structure enables us to develop an efficient verification technique in case of $B$ being positive definite.
Yamamoto~\cite{Yamamoto1984JJAM} and Rump~\cite{Rump2013JCAM} derived componentwise and normwise bounds, respectively, of the error of the approximate solution.
See also \cite{Rump2010}.
These methods need a numerically computed inverse of the coefficient matrix, whereas the proposed technique does not need such a numerical inverse, and instead needs a lower bound of the smallest eigenvalue of $B$.
\par
In the rest of the paper, we use the following notations:
For a real matrix  $A = (a_{ij}) \in \mathbb{R}^{m \times n}$, a nonnegative matrix consisting of entrywise absolute values is denoted by $|A| = (|a_{ij}|)$.
For $B = (b_{ij}) \in \mathbb{R}^{m \times n}$ and $\alpha \in \mathbb{R}$, the inequality $A < B$ means $a_{ij} < b_{ij}$ holds entrywise and the inequality $A < \alpha$ means $a_{ij} < \alpha$ holds entrywise.
\par
The rest of this paper is organized as follows: In Section~\ref{sec:bSSHK}, we briefly review the block Sakurai--Sugiura Hankel method and its error analysis derived by Miyata et al.~\cite{MiyataDuSogabeYamamotoZhang2009}.
Thanks to this result, we derive a computable rigorous error bound for complex moment in Section~\ref{sec:bound}.
We also put several remarks on the implementation of our method in Section~\ref{sec:implementation}.
In Section~\ref{sec:examples}, we show two numerical examples illustrating the performance of our method.
In Section~\ref{sec:conclusions}, we conclude the paper for discussing potentials of our method for parallel implementation and future directions.
\section{Block Hankel-type Sakurai--Sugiura method} \label{sec:bSSHK}
We review the block Sakurai--Sugiura Hankel method \cite{IkegamiSakuraiNagashima2010}, which is the basis of the proposed method.
The block Sakurai--Sugiura Hankel method has parameters such as the block size $L \in \mathbb{N}_+$, the order of moment $M \in \mathbb{N}_+$, a random matrix $V \in \mathbb{C}^{n \times L}$ whose column vectors consist of a linear combination of all eigenvectors, the basis vectors of the kernel of $B$, say $\Ker\,B$, and the scaling parameters $(\gamma, \rho) \in \mathbb{R} \times \mathbb{R}$ for the eigenvalues.
The $p$\,th complex moment matrix is given by
\begin{equation}
	\mathsf{M}_p = \frac{1}{2 \pi \mathrm{i}} \oint_\Gamma \left(z-\gamma\right)^p V^\mathsf{H} B (z B - A)^{-1} B V \mathrm{d} z, \quad p = 0, 1, 2, \dots, 2 M - 1	\label{eq:moment}
\end{equation}
defined on the closed Jordan curve $\Gamma$ through the end points of the interval $\Omega = [a, b]$, where $\mathrm{i}=\sqrt{-1}$ is the imaginary unit and $\pi$ is the circle ratio.
Denote the block Hankel matrices consisting of the moments \eqref{eq:moment} by 
\begin{align*}
	H_M^< & = 
	\begin{bmatrix}
		\mathsf{M}_1   & \mathsf{M}_2     & \cdots & \mathsf{M}_{M}   \\
		\mathsf{M}_2   & \mathsf{M}_3     & \cdots & \mathsf{M}_{M+1} \\
		\vdots  & \vdots    & \ddots & \vdots    \\
		\mathsf{M}_{M} & \mathsf{M}_{M+1} & \cdots & \mathsf{M}_{2M-1}
	\end{bmatrix}
	\in \mathbb{C}^{LM \times LM}, \\
	H_M & = 
	\begin{bmatrix}	
		\mathsf{M}_0     & \mathsf{M}_1   & \cdots & \mathsf{M}_{M-1}  \\
		\mathsf{M}_1     & \mathsf{M}_2   & \cdots & \mathsf{M}_{M}    \\
		\vdots    & \vdots  & \ddots & \vdots     \\
		\mathsf{M}_{M-1} & \mathsf{M}_{M} & \cdots & \mathsf{M}_{2M-2}
	\end{bmatrix}\in \mathbb{C}^{LM \times LM}.
\end{align*}
Then, the following theorem show that the block Sakurai--Sugiura Hankel method can compute eigenvalues in a prescribed domain and their corresponding eigenvectors \cite[Theorems~5 and 6]{IkegamiSakuraiNagashima2010}.
\begin{theorem}\label{th:Hankel}
	Let an eigenvalue and the corresponding eigenvector of the regular part of the matrix pencil $zH_M - H_M^<$ be denoted by $\theta_i$ and $\boldsymbol{u}_i$, respectively.
	Let 
	\begin{equation*}
		S_p = \frac{1}{2 \pi \mathrm{i}} \oint_\Gamma (z-\gamma)^p (z B - A)^{-1} B V \mathrm{d}z
	\end{equation*}
	and $S = [S_0, S_1, \dots, S_{M-1}]$.
	If $\mathrm{rank}(H_M) = m$ holds, then the eigenvalue in $\Gamma$ and the corresponding eigenvector of $\eqref{eq:gevp}$ are given by $\lambda_i = \gamma + \theta_i$ and $\boldsymbol{x}_i = S \boldsymbol{u}_i$ ($i = 1, 2, \dots, m$), respectively.	
\end{theorem}

We remark that the condition $\mathrm{rank}(H_M) = m$ implies $LM \ge m$.

Next, we give a relationship between the target eigencomponents in the columns of $V$ and the rank of $H_M$.
%%%
%\begin{rem}
Recall the Weierstrass canonical form of the matrix pencil $zB - A$ \cite[Proposition~7.8.3]{Bernstein2018}.
There exists a nonsingular matrix $X \in \mathbb{C}^{n \times n}$ such that $X^\mathsf{H} (zB - A) X = z \mathrm{I}_0 - \Lambda$, where $\mathrm{I}_0$ is a diagonal matrix whose leading $r$ diagonal entries are one and whose trailing $n-r$ diagonal entries are zeros, and $\Lambda$ is a diagonal matrix whose leading $r$ diagonal entries are the eigenvalues of \eqref{eq:gevp} and whose trailing $n-r$ diagonal entries are one.
Note that the columns of $X = [\boldsymbol{x}_1, \boldsymbol{x}_2, \dots, \boldsymbol{x}_n]$ are the appropriately scaled eigenvectors of matrix pencil $z B - A$, where $\boldsymbol{x}_1$, $\boldsymbol{x}_2$, \dots $\boldsymbol{x}_r \in \mathbb{C}^n$ correspond to the eigenvalues $\lambda_1$, $\lambda_2$, \dots, $\lambda_r \in \mathbb{R}$, respectively, and $\boldsymbol{x}_k$, $\boldsymbol{x}_{k+1}$, \dots $\boldsymbol{x}_n \in \mathbb{C}^n$ form a basis of $\Ker\,B$.
Then, from $X^\mathsf{H} B X = \mathrm{I}_0$ and the residue theorem, the complex moment \eqref{eq:moment} is expressed as
\begin{align}\label{eq:mu_p}
	\mathsf{M}_p &= \frac{1}{2 \pi \mathrm{i}} \oint_\Gamma \left(z-\gamma\right)^p V^\mathsf{H} B X (z\mathrm{I}_0 - \Lambda)^{-1} \mathrm{I}_0 X^{-1} B V \mathrm{d} z \\
	& = \frac{1}{2 \pi \mathrm{i}} \oint_\Gamma \left(z-\gamma\right)^p \sum_{k=1}^{r}\left(\frac{V^\mathsf{H}B\boldsymbol{x}_k\boldsymbol{x}_k^\mathsf{H}BV}{z-\lambda_k}\right)\mathrm{d} z\\
	& = \sum_{k=1}^m(\lambda_k-\gamma)^p \mathcal{V}_k,
\end{align} 
where $\mathcal{V}_k = V^\mathsf{H} B \boldsymbol{x}_k \boldsymbol{x}_k^\mathsf{H} B V \in \mathbb{C}^{L \times L}$.
This is represented by
\[
\mathsf{M}_p = V^\mathsf{H} B X_{\Omega}\Lambda_{\Omega}^p X_{\Omega}^\mathsf{H} B V,\quad X_{\Omega}=[\boldsymbol{x}_1, \boldsymbol{x}_2, \dots, \boldsymbol{x}_m],\quad \Lambda_{\Omega}=\mathrm{diag}(\lambda_1-\gamma,\lambda_2-\gamma,\dots,\lambda_m-\gamma).
\]
Using this form, we have
\begin{align*}
	H_M&=\begin{bmatrix}
		V^\mathsf{H} B X_{\Omega} X_{\Omega}^\mathsf{H} B V     & V^\mathsf{H} B X_{\Omega}\Lambda_{\Omega} X_{\Omega}^\mathsf{H} B V   & \cdots & V^\mathsf{H} B X_{\Omega}\Lambda_{\Omega}^{M-1} X_{\Omega}^\mathsf{H} B V  \\
		V^\mathsf{H} B X_{\Omega}\Lambda_{\Omega} X_{\Omega}^\mathsf{H} B V     & V^\mathsf{H} B X_{\Omega}\Lambda_{\Omega}^2 X_{\Omega}^\mathsf{H} B V   & \cdots & V^\mathsf{H} B X_{\Omega}\Lambda_{\Omega}^{M} X_{\Omega}^\mathsf{H} B V    \\
		\vdots    & \vdots  & \ddots & \vdots     \\
		V^\mathsf{H} B X_{\Omega}\Lambda_{\Omega}^{M-1} X_{\Omega}^\mathsf{H} B V & V^\mathsf{H} B X_{\Omega}\Lambda_{\Omega}^{M} X_{\Omega}^\mathsf{H} B V & \cdots & V^\mathsf{H} B X_{\Omega}\Lambda_{\Omega}^{2M-2} X_{\Omega}^\mathsf{H} B V
	\end{bmatrix}\\
	&=\begin{bmatrix}
		V^\mathsf{H} B X_{\Omega} X_{\Omega}^\mathsf{H}\\
		V^\mathsf{H} B X_{\Omega}\Lambda_{\Omega} X_{\Omega}^\mathsf{H}\\
		\vdots\\
		V^\mathsf{H} B X_{\Omega}\Lambda_{\Omega}^{M-1} X_{\Omega}^\mathsf{H}
	\end{bmatrix}B\begin{bmatrix}
		X_{\Omega} X_{\Omega}^\mathsf{H} B V     & X_{\Omega}\Lambda_{\Omega} X_{\Omega}^\mathsf{H} B V   & \cdots & X_{\Omega}\Lambda_{\Omega}^{M-1} X_{\Omega}^\mathsf{H} B V
	\end{bmatrix}\\
	&=S^{\mathsf{H}}BS,
\end{align*}
%	$S = [S_0, S_1, \dots, S_{M-1}]$
where
\begin{align*}
	S & = [S_0, S_1, \dots, S_{M-1}], \\
	S_p & = \frac{1}{2 \pi \mathrm{i}} \oint_\Gamma (z-\gamma)^p (z B - A)^{-1} B V \mathrm{d}z=X_{\Omega}\Lambda^p_{\Omega}X_{\Omega}^HBV=\sum_{k=1}^m(\lambda_k-\gamma)^p \boldsymbol{x}_k \boldsymbol{x}_k^\mathsf{H} B V.
\end{align*}
Meanwhile, it follows that the range of $S$ satisfies
\[
R(S)=\bigoplus_{p=0}^{M-1}R(S_p)\subset{\mathrm{span}}\left\{\boldsymbol{x}_k:\boldsymbol{x}_k^\mathsf{H}BV\neq 0,\quad k=1,2,\dots,m \right\}.
\]
This implies
\[
\dim\left(R(S)\right)\le \dim\left({\mathrm{span}}\left\{\boldsymbol{x}_k:\boldsymbol{x}_k^\mathsf{H}BV\neq 0,\quad k=1,2,\dots,m \right\}\right).
\]
If we set $V$ such that $\boldsymbol{x}_k^\mathsf{H}BV= 0$ for some $k=1,2,\dots,m$, then $\dim\left(R(S)\right)<m$ and 
the Hankel matrix $H_M$ becomes singular.
%\end{rem}

\par
In practice, the method uses the $N$-point trapezoidal rule to approximate the complex moment \eqref{eq:moment} multiplied by $\rho^{-(p+1)}$.
We take a domain of integration $\Gamma$ in \eqref{eq:moment} as the circle
\begin{equation}
	\Gamma = \{z \in \mathbb{C}| z = \gamma + \rho \mathrm{e}^{ \mathrm{i} \theta}, \theta \in \mathbb{R} \}, \quad
	\gamma = \frac{b+a}{2}, \quad
	\rho = \frac{b-a}{2}
	\label{eq:circle}
\end{equation}
and approximate the complex moment \eqref{eq:moment} with the following equi-distributed quadrature points:
\begin{equation}
	z_j=\gamma+\rho \mathrm{e}^{\mathrm{i}\theta_j},\quad \theta_j = \frac{2j - 1}{N} \pi, \quad j=1, 2, \dots, N.
	\label{eq:quadpoint}
\end{equation}
%\par
%Suppose that $\Gamma$ is a circle with radius $\rho > 0$ with center at $\gamma \in \mathbb{R}$, as defined in \eqref{eq:circle}.
%The variable change with $z=\gamma+\rho \mathrm{e}^{\mathrm{i}\theta}$ in the integral \eqref{eq:mu_p} follows.
%Letting $\mathsf{M}_p^{(N)} = \rho^{p+1} \hat{\mathsf{M}}_p^{(N)}$ in \eqref{eq:moment_n}, 
We review the error analysis in \cite{MiyataDuSogabeYamamotoZhang2009} to derive a rigorous error bound of the complex moment \eqref{eq:moment} in Section~\ref{sec:bound}.
The trapezoidal rule with the equi-distributed quadrature points \eqref{eq:quadpoint} approximates the complex moment \eqref{eq:moment} as
\begin{align}
	\mathsf{M}_p^{(N)} 
	= \frac{1}{N}\sum_{j=1}^{N}\left(\rho \mathrm{e}^{\mathrm{i}\theta_j}\right)^{p+1}\left(\sum_{k=1}^{r}\frac{\mathcal{V}_k}{\rho \mathrm{e}^{\mathrm{i}\theta_j}-(\lambda_k-\gamma)} \right) = \sum_{k=1}^r \mathcal{V}_k \left(\frac{1}{N}\sum_{j=1}^{N}\rho^{p}\mathrm{e}^{\mathrm{i}p\theta_j}\frac{\rho \mathrm{e}^{\mathrm{i}\theta_j}}{\rho \mathrm{e}^{\mathrm{i}\theta_j}-(\lambda_k-\gamma)}\right). \label{eq:hat_mup}
\end{align}
Since the number of eigenvalues inside $\Gamma$ is $m$, $\left| (\lambda_k-\gamma) / \rho \right| < 1$ holds for $k = 1$, $2$, $\dots$, $m$.
Noting the sum of geometric series, the quantity in the parentheses in \eqref{eq:hat_mup} for $k=1$, $2$, $\dots$, $m$ is written as 
\begin{align}
	\frac{1}{N}\sum_{j=1}^{N}\rho^{p}\mathrm{e}^{\mathrm{i}p\theta_j}\frac{\rho \mathrm{e}^{\mathrm{i}\theta_j}}{\rho \mathrm{e}^{\mathrm{i}\theta_j}-(\lambda_k-\gamma)} & = \frac{1}{N} \sum_{j=1}^{N}\rho^{p}\mathrm{e}^{\mathrm{i}p\theta_j}\left(\sum_{\ell=0}^{\infty}\left(\frac{\lambda_k-\gamma}{\rho \mathrm{e}^{\mathrm{i}\theta_j}}\right)^\ell\right) \notag \\
	& = \sum_{\ell=0}^\infty\rho^{p-\ell}(\lambda_k-\gamma)^\ell\left(\frac{1}{N}\sum_{j=1}^{N}\mathrm{e}^{\mathrm{i}(p-\ell)\theta_j}\right)\nonumber\\
	& = \sum_{s=0}^{\infty}\rho^{-sN}(\lambda_k-\gamma)^{p+sN} = (\lambda_k-\gamma)^p\left(\frac{1}{1-\left(\frac{\lambda_k-\gamma}{\rho}\right)^N}\right). \label{eq:mu_in}
\end{align}
Here, we set $p-\ell = -sN$ ($s = 0$, $1$, $2$, $\dots$), due to the property
\begin{equation*}
	\frac{1}{N}\sum_{j=1}^{N}\mathrm{e}^{\mathrm{i}h\theta_j} = 
	\begin{cases}
		1 \quad (h\in N\mathbb{Z}),\\
		0 \quad (\mathrm{otherwise}).
	\end{cases}
\end{equation*}
The other $r-m$ eigenvalues $\lambda_k$ ($k = m+1$, $m+2$, $\dots$, $r$) outside the domain $\Gamma$ satisfy the inequalities $\left| \rho  / (\lambda_k-\gamma) \right|<1$.
Noting the sum of geometric series, the quantity in the parentheses in \eqref{eq:hat_mup} for $k = m+1, m+2, \dots, r$ is written as 
\begin{align}
	\frac{1}{N}\sum_{j=1}^{N}\rho^{p}\mathrm{e}^{\mathrm{i}p\theta_j}\frac{\rho \mathrm{e}^{\mathrm{i}\theta_j}}{\rho \mathrm{e}^{\mathrm{i}\theta_j}-(\lambda_k-\gamma)} 
	& = \frac{1}{N}\sum_{j=1}^{N}\rho^{p}\mathrm{e}^{\mathrm{i}p\theta_j}\left(-\frac{\frac{\rho \mathrm{e}^{\mathrm{i}\theta_j}}{\lambda_k-\gamma}}{1-\frac{\rho\mathrm{e}^{\mathrm{i}\theta_j}}{\lambda_k-\gamma}}\right) \notag \\
	& = \frac{1}{N}\sum_{j=1}^{N}\rho^{p}\mathrm{e}^{\mathrm{i}p\theta_j}\left(-\sum_{\ell=0}^{\infty}\left(\frac{\rho \mathrm{e}^{\mathrm{i}\theta_j}}{\lambda_k-\gamma}\right)^{\ell+1}\right)\nonumber\\
	& =\sum_{\ell=0}^\infty-\rho^{p+\ell+1}(\lambda_k-\gamma)^{-(\ell+1)}\left(\frac{1}{N}\sum_{j=1}^{N}\mathrm{e}^{\mathrm{i}(p+\ell+1)\theta_j}\right) \notag\\
	& = \sum_{s=1}^{\infty}-\rho^{sN}(\lambda_k-\gamma)^{-(sN-p)}\nonumber\\
	& = (\lambda_k-\gamma)^p\left(\frac{-\left(\frac{\rho}{\lambda_k-\gamma}\right)^N}{1-\left(\frac{\rho}{\lambda_k-\gamma}\right)^N}\right). \label{eq:mu_out}
\end{align}
Here, we set $p+\ell+1 = sN$ ($s = 1, 2, \dots$).
It follows from \eqref{eq:hat_mup}, \eqref{eq:mu_in}, and \eqref{eq:mu_out} that the approximated complex moment is split into two parts $\mathsf{M}_{p}^{(N)} = \mathsf{M}_{p,\mathrm{in}}^{(N)} + \mathsf{M}_{p,\mathrm{out}}^{(N)}$, where
\begin{equation}
	\mathsf{M}_{p,\mathrm{in}}^{(N)} = \sum_{k=1}^m  (\lambda_k-\gamma)^p\left(\frac{1}{1-\left(\frac{\lambda_k-\gamma}{\rho}\right)^N}\right) \mathcal{V}_k, \quad 
	\mathsf{M}_{p,\mathrm{out}}^{(N)} = \sum_{k=m+1}^r (\lambda_k-\gamma)^p\left(\frac{-\left(\frac{\rho}{\lambda_k-\gamma}\right)^N}{1-\left(\frac{\rho}{\lambda_k-\gamma}\right)^N}\right) \mathcal{V}_k
	\label{eq:mu_in/out}
\end{equation}
are regarding the inside and outside of $\Gamma$, respectively.
Together with \eqref{eq:mu_p}, we have the truncation error analysis of the $N$-point trapezoidal rule $\mathsf{M}_{p}^{(N)} - \mathsf{M}_p$.

\section{Error bound of the complex moment} \label{sec:bound}
Based on the error analysis in the previous section, we derive a rigorous error bound for each complex moment $\mathsf{M}_p$.
Let 
\begin{equation*}
\alpha_k = \frac{1}{1-\left(\frac{\lambda_k-\gamma}{\rho}\right)^N}, \quad k=1,2,\dots,m, \qquad	\beta_k = \frac{-\left(\frac{\rho}{\lambda_k-\gamma}\right)^N}{1-\left(\frac{\rho}{\lambda_k-\gamma}\right)^N}, \quad k=m+1, m+2, \dots, r.
\end{equation*}
Then, the rightmost sides of \eqref{eq:mu_in} and \eqref{eq:mu_out} become $(\lambda_k-\gamma)^p\alpha_k$ ($k=1, 2, \dots, m$) and $(\lambda_k-\gamma)^p\beta_k$ ($k=m+1,\dots,n$), respectively.
Then, we simplify the expressions of the approximated complex moment \eqref{eq:mu_in/out}
%\begin{equation}
%\mathsf{M}_p-\mathsf{M}_{p}^{(N)} = \sum_{k=1}^m (\lambda_k-\gamma)^p \left(\frac{-\left(\frac{\lambda_k-\gamma}{\rho}\right)^N}{1-\left(\frac{\lambda_k-\gamma}{\rho}\right)^N}\right) \mathcal{V}_k +\sum_{k=m+1}^r (\lambda_k-\gamma)^p\left(\frac{\left(\frac{\rho}{\lambda_k-\gamma}\right)^N}{1-\left(\frac{\rho}{\lambda_k-\gamma}\right)^N}\right) \mathcal{V}_k. \label{eq:err}
%\end{equation}
\begin{equation*}
\mathsf{M}_{p}^{(N)}  = \mathsf{M}_{p,\mathrm{in}}^{(N)}+\mathsf{M}_{p,\mathrm{out}}^{(N)} = \sum_{k=1}^m (\lambda_k-\gamma)^p\alpha_k \mathcal{V}_k + \sum_{k=m+1}^r (\lambda_k-\gamma)^p\beta_k \mathcal{V}_k.
\end{equation*}
The truncation error is given by
\begin{equation*}
\mathsf{M}_p-\mathsf{M}_{p}^{(N)}  = \sum_{k=1}^m  (\lambda_k-\gamma)^p(1-\alpha_k) \mathcal{V}_k - \sum_{k=m+1}^r (\lambda_k-\gamma)^p \beta_k \mathcal{V}_k.
\end{equation*}

We note that the following identities of the eigenvalues of a Hankel matrix pencil are useful for our verification methods.
\begin{lemma}
	Assume that $\mathrm{rank}(H_M) = m$ holds.
	Then, the Hankel matrix pencil $z H_M - H_M^<$ consisting of $\mathsf{M}_p$ and the Hankel matrix pencil $z H_{M,\mathrm{in}}^{(N)} - H_{M,\mathrm{in}}^{<, (N)}$ with
	\begin{align*}
	H_{M,\mathrm{in}}^{<, (N)} =
	\begin{bmatrix}
	\mathsf{M}_{1,{\mathrm{in}}}^{(N)} & \mathsf{M}_{2,{\mathrm{in}}}^{(N)} & \cdots & \mathsf{M}_{M,{\mathrm{in}}}^{(N)} \\
	\mathsf{M}_{2,{\mathrm{in}}}^{(N)} & \mathsf{M}_{3,{\mathrm{in}}}^{(N)} & \cdots & \mathsf{M}_{M+1,{\mathrm{in}}}^{(N)}\\
	\vdots  & \vdots    & \ddots & \vdots    \\
	\mathsf{M}_{M,{\mathrm{in}}}^{(N)} & \mathsf{M}_{M+1,{\mathrm{in}}}^{(N)}& \cdots & \mathsf{M}_{2M-1,{\mathrm{in}}}^{(N)}
	\end{bmatrix}
	\in \mathbb{C}^{LM \times LM}, \\
	H_{M,\mathrm{in}}^{(N)} = 
	\begin{bmatrix}	
	\mathsf{M}_{0,{\mathrm{in}}}^{(N)} & \mathsf{M}_{1,{\mathrm{in}}}^{(N)} & \cdots & \mathsf{M}_{M-1,{\mathrm{in}}}^{(N)} \\
	\mathsf{M}_{1,{\mathrm{in}}}^{(N)} & \mathsf{M}_{2,{\mathrm{in}}}^{(N)} & \cdots & \mathsf{M}_{M,{\mathrm{in}}}^{(N)}   \\
	\vdots    & \vdots  & \ddots & \vdots     \\
	\mathsf{M}_{M-1,{\mathrm{in}}}^{(N)} & \mathsf{M}_{M,{\mathrm{in}}}^{(N)} & \cdots & \mathsf{M}_{2M-2,{\mathrm{in}}}^{(N)}
	\end{bmatrix}
	\in \mathbb{C}^{LM \times LM}
	\end{align*}
	consisting of $\mathsf{M}_{p,\mathrm{in}}^{(N)}$ have the same eigenvalues.
\end{lemma}
\begin{proof}
	Let $V=[\boldsymbol{v}_1,\boldsymbol{v}_2,\dots,\boldsymbol{v}_L]$, $\boldsymbol{v}_i=\sum_{j = 1}^nc_j\boldsymbol{x}_j$ and $V'=[\boldsymbol{v}'_1,\boldsymbol{v}'_2,\dots,\boldsymbol{v}'_L]$, $\boldsymbol{v}'_i=\sum_{j = 1}^n\alpha_j^{1/2}c_j\boldsymbol{x}_j$.
	%	Since that Theorem~\ref{th:Hankel} holds irrespective of the length of the eigenvector $\boldsymbol{x}_i$ in the Weierstrass decomposition of $z B -A$, i.e., the scaler multiple of $\mathcal{V}_k$ can be viewed as a scaling of the eigenvector $\boldsymbol{x}_k$ with ${\alpha_k}^{1/2}$
	Then, we have the equalities
	\begin{equation*}
	\alpha_k \mathcal{V}_k = \alpha_kV^\mathsf{H} B \boldsymbol{x}_k\boldsymbol{x}_k^\mathsf{H} B V = \alpha_k\left(c_k\boldsymbol{e}_k\right)\left(c_k\boldsymbol{e}_k\right)^\mathsf{H}=\left(\alpha_k^{1/2}c_k\boldsymbol{e}_k\right)\left(\alpha_k^{1/2}c_k\boldsymbol{e}_k\right)^\mathsf{H}={V'}^\mathsf{H} B \boldsymbol{x}_k\boldsymbol{x}_k^\mathsf{H} B V'
	\end{equation*}
	for $k = 1$, $2$, \dots, $m$.
	Since Theorem~\ref{th:Hankel} holds irrespective of the scaling regarding the eigenvectors in the columns of $V$, the lemma holds.
\end{proof}
Hence, we derive an enclosure of $\mathsf{M}_{p,\mathrm{in}}^{(N)}$ instead of an enclosure of $\mathsf{M}_p$.
We can enclose $\mathsf{M}_{p,\mathrm{in}}^{(N)}$ by using the quantity $|\mathsf{M}_{p,\mathrm{out}}^{(N)}|$ and computing the truncated complex moment $\mathsf{M}_{p}^{(N)}$ with interval arithmetic.
Let us denote a numerical approximation of $\mathsf{M}_{p}^{(N)}$ by $\tilde{\mathsf{M}}_{p}^{(N)}$.
Hereafter, we denote a numerically computed quantity that may suffer from rounding errors with a tilde.
Then, it follows from $\mathsf{M}_{p}^{(N)}-\mathsf{M}_{p,\mathrm{in}}^{(N)} =\mathsf{M}_{p,\mathrm{out}}^{(N)}$ that the inequality
\begin{align*}
\left|\mathsf{M}_{p,\mathrm{in}}^{(N)}-\tilde{\mathsf{M}}_{p}^{(N)}\right| \leq \left|\mathsf{M}_{p,\mathrm{in}}^{(N)}-\mathsf{M}_{p}^{(N)}\right|+ \left|\mathsf{M}_{p}^{(N)}-\tilde{\mathsf{M}}_{p}^{(N)}\right| = \left|\mathsf{M}_{p,\mathrm{out}}^{(N)}\right| + \left|\mathsf{M}_{p}^{(N)}-\tilde{\mathsf{M}}_{p}^{(N)}\right|
\end{align*}
holds.
Let us denote the interval matrix with radius $r\in \mathbb{R}_{+}^{L\times L}$ centered at $c \in \mathbb{C}^{L\times L}$ by $\langle c, r \rangle$.
To sum up the above discussion, we have the following theorem:
\begin{theorem}\label{thm:errbound}
	The computable rigorous enclosure of $\mathsf{M}_{p,\mathrm{in}}^{(N)}$ is given by
	\begin{align}\label{eq:rigorous_enclosure}
	\mathsf{M}_{p,\mathrm{in}}^{(N)} \in \left\langle\mathsf{M}_{p}^{(N)},\left|\mathsf{M}_{p,\mathrm{out}}^{(N)}\right|\right \rangle \subset \left \langle \tilde{\mathsf{M}}_{p}^{(N)},\left|\mathsf{M}_{p,\mathrm{out}}^{(N)}\right|+\left|\mathsf{M}_{p}^{(N)}-\tilde{\mathsf{M}}_{p}^{(N)}\right|\right\rangle.
	\end{align}
\end{theorem}
The proof is already completed by the above discussions.
We can enclose $\left|\mathsf{M}_{p}^{(N)}-\tilde{\mathsf{M}}_{p}^{(N)}\right|$ using standard verification methods using interval arithmetic, whereas the complex moment $\mathsf{M}_{p, \mathrm{out}}^{(N)}$ regarding the outside of $\Gamma$ is bounded as follows:
\begin{theorem}\label{th:outer_mu}
	Let $V \in \mathbb{C}^{n \times L}$ be an arbitrary matrix such that
	\begin{equation*}
	V=XC=\left[ X_0, X_1 \right]
	\begin{bmatrix}
	C_0 \\
	C_1
	\end{bmatrix},
	\end{equation*}
	where the columns of $X_0 = \left[\boldsymbol{x}_{r+1},\boldsymbol{x}_{r+2},\dots,\boldsymbol{x}_n\right]$ form a basis of $\Ker\,B$ ($r=\mathrm{rank}(B)$), the columns of $X_1=\left[\boldsymbol{x}_1,\boldsymbol{x}_2,\dots,\boldsymbol{x}_r\right]$ form a basis of $\Ker \, B^\perp$, $C_0 \in \mathbb{C}^{(n-r) \times L}$, and $C_1 \in \mathbb{C}^{r \times L}$ has at least one nonzero entry in each column and each row.
	Suppose $N > 2M-1 \ge p$ and that $\hat{\lambda}$ satisfies $\left|\hat{\lambda}-\gamma\right|=\min_{k=m+1,m+2,\dots,r}\left|\lambda_k-\gamma\right|$.
	Then, the complex moment \eqref{eq:mu_in/out} is bounded above as 
	\begin{align}
	\left|\mathsf{M}_{p,\mathrm{out}}^{(N)}\right| \le (r-m)\left|\hat{\lambda}-\gamma\right|^{p}\left(\frac{\left(\frac{\rho}{\left|\hat{\lambda}-\gamma\right|}\right)^N}{1-\left(\frac{\rho}{\left|\hat{\lambda}-\gamma\right|}\right)^N}\right)\left\|V^{\mathsf{H}}BV\right\|_{\mathsf{F}}
	\label{eq:outer_mu}
	\end{align}
	for $p = 0$, $1$, \dots, $2M-1$, where $\| \cdot \|_{\mathsf{F}}$ denotes the Frobenius norm.
\end{theorem}
\begin{proof}
	Regarding the fraction factor in \eqref{eq:mu_in/out} as the geometric series, we have
	\begin{align*}
	\left|\mathsf{M}_{p,\mathrm{out}}^{(N)}\right|&=\left|\sum_{k=m+1}^r (\lambda_k-\gamma)^p\left(\sum_{s=1}^\infty\left(\frac{\rho}{\lambda_k-\gamma}\right)^{sN}\right) \mathcal{V}_k \right| \le\sum_{k=m+1}^r\left(\sum_{s=1}^\infty\rho^{sN}\left|\lambda_k-\gamma\right|^{-(sN-p)}\right) \left| \mathcal{V}_k \right| \\
	&\le\sum_{k=m+1}^r \left(\sum_{s=1}^\infty\rho^{sN}\left|\hat{\lambda}-\gamma\right|^{-(sN-p)}\right) \left| \mathcal{V}_k \right| = \sum_{k=m+1}^r \left|\hat{\lambda}-\gamma\right|^{p}\left(\frac{\left(\frac{\rho}{\left|\hat{\lambda}-\gamma\right|}\right)^N}{1-\left(\frac{\rho}{\left|\hat{\lambda}-\gamma\right|}\right)^N}\right) \left| \mathcal{V}_k \right|.
	\end{align*}
	Note that the property $B V = B X_1 C_1 + B X_0 C_0 = B X_1 C_1$ gives
	\begin{equation*}
	V^{\mathsf{H}} B V = V^{\mathsf{H}} B X_1 C_1 = (BV)^{\mathsf{H}} X_1 C_1=(B X_1 C_1)^{\mathsf{H}} X_1 C_1 = C_1^{\mathsf{H}} X_1^{\mathsf{H}} B X_1 C_1 = C_1^{\mathsf{H}} C_1.
	\end{equation*}
	Hence, we have 
	\begin{align*}
	\left|\mathcal{V}_k\right| & \le \left\| V^\mathsf{H} B \boldsymbol{x}_k \boldsymbol{x}_k^\mathsf{H} B V\right\|_{\mathsf{F}} = 
	\left\|C_1^{\mathsf{H}}X_1^{\mathsf{H}}B\boldsymbol{x}_k \boldsymbol{x}_k^\mathsf{H} B X_1 C_1\right\|_{\mathsf{F}} =\left\|C_1^{\mathsf{H}}\boldsymbol{e}_k \boldsymbol{e}_k^\mathsf{T} C_1\right\|_{\mathsf{F}} \le \left\| C_1^{\mathsf{H}}C_1\right\|_{\mathsf{F}}\\
	& =\left\|V^{\mathsf{H}}BV\right\|_{\mathsf{F}}, \quad k=1,2,\dots,r,
	\end{align*}
	where $\boldsymbol{e}_k$ is the $k$\,th standard basis vector of $\mathbb{R}^n$, i.e., the $k$\,th entry is one and the remaining entries are zero.
	Therefore, we obtain \eqref{eq:outer_mu}.
\end{proof}

\section{Implementation}\label{sec:implementation}
In this section, we present an implementation of the block Sakurai--Sugiura Hankel-based method for numerically verifying the partial eigenvalues $\lambda_i \in \Omega$, $i = 1$, $2$, \dots, $m$.
Suppose that the number of the eigenvalues  in $\Gamma$ is $m$.
We set $L$ and $M$ such that $m = L M$.
{Note that if $m$ is a prime number, either $L$ or $M$ must be one and the other must be $m$.}
%% TODO
%Then, we have the the complex moment \eqref{eq:moment} multiplied by $\rho^{-(p+1)}$ 
%\begin{equation*}
%\frac{\mathsf{M}_p}{\rho^{p+1}}=\frac{1}{2 \pi\rho \mathrm{i}} \oint_\Gamma \left(\frac{z-\gamma}{\rho}\right)^p V^\mathsf{H}B (zB - A)^{-1}BV \mathrm{d} z
%= \frac{1}{2 \pi} \int_0^{2 \pi} \mathrm{e}^{\mathrm{i}(p+1)\theta} V^\mathsf{H}B ( (\gamma + \rho\mathrm{e}^{ \mathrm{i} \theta})B - A)^{-1} BV \mathrm{d} \theta
%\end{equation*}
%and thus its approximated complex moment 
%\begin{equation}
%\hat{\mathsf{M}}_p^{(N)} = \frac{1}{N} \sum_{j=1}^N
%\mathrm{e}^{ \mathrm{i}(p+1)\theta_j} V^\mathsf{H}B ( (\gamma + \rho\mathrm{e}^{ \mathrm{i} \theta_j}) B - A)^{-1} BV.
%\label{eq:moment_n}
%\end{equation}
%As described above, the target eigenvalues $\lambda_i \in \Omega$ are computed by solving the small eigenproblem of the Hankel matrix pencil $zH_{M,\mathrm{in}}^{(N)} - H_{M,\mathrm{in}}^{<, (N)}$ consisting of $\mathsf{M}_{p,\mathrm{in}}^{(N)}$.
To rigorously enclose the eigenvalues, we verify each block $\mathsf{M}_{p,\mathrm{in}}^{(N)}$ of the block Hankel matrices by using Theorem~\ref{thm:errbound},  and then apply the verified
eigenvalue computation methods \cite{Miyajima2012JCAM,Rump1989} to the small eigenproblem of regular Hankel matrix pencil consisting of $\mathsf{M}_{p,\mathrm{in}}^{(N)}$.
The matrix $\left|\mathsf{M}_{p,\mathrm{out}}^{(N)}\right|$ in \eqref{eq:rigorous_enclosure} can be bounded by using \eqref{eq:outer_mu}.
The number of quadrature points can be automatically determined from the error bound \eqref{eq:outer_mu} by
\begin{equation}\label{eq:setN}
N \ge \frac{\log\left(\frac{\delta}{c+\delta}\right)}{\log \frac{\rho}{\left|\hat{\lambda}-\gamma\right|}},\quad c=(r-m) \left\|V^{\mathsf{H}}BV\right\|_{\mathsf{F}} \max_{p=1,2, \dots,2M-1}\left|\hat{\lambda}-\gamma\right|^{p},
\end{equation}
where $\delta$ denotes the tolerance of  quadrature error.
Hence, there is a trade-off between the accuracy for the quadrature and the central processing unit (CPU) time.

\par
The matrix $\left|\mathsf{M}_{p}^{(N)}-\tilde{\mathsf{M}}_{p}^{(N)}\right|$ in \eqref{eq:rigorous_enclosure} can be also bounded by evaluating the numerical error.
To rigorously bound the numerical error, we need verification of a numerical solution of the linear system with multiple right-hand sides, that is $(z_j B - A) Y_j = B V$, which comes from
\begin{equation*}
\mathsf{M}_p^{(N)} 
%= \frac{1}{N} \sum_{j=1}^N V^\mathsf{H} B (\rho \mathrm{e}^{\mathrm{i} \theta_j})^{p+1} (z_j B - A)^{-1} B V
=\frac{1}{N} \sum_{j=1}^N V^\mathsf{H} B (\rho \mathrm{e}^{\mathrm{i} \theta_j})^{p+1} Y_j^*,\quad Y_j^* =(z_j B - A)^{-1} B V.
\end{equation*}
The enclosure of $Y_j^*$ can be obtained by standard verification methods, e.g., \cite{Rump2013JCAM}, whereas we consider efficiently enclosing the solution $Y_j^*$ for positive definite $B$.

\begin{theorem}\label{thm:efficientVNC}
	Let $A$ be a Hermitian matrix and $B$ a Hermitian positive definite matrix.
	The quadrature points $z_j$, $j = 1$, $2$, \dots, $N$ are defined as in \eqref{eq:quadpoint}.
	Denote the $i$\,th entries of the solution $\boldsymbol{y}^\ast = (z_j B - A)^{-1} \boldsymbol{b}$ and an approximate solution $\tilde{\boldsymbol{y}}$ of $(z_j B - A) \boldsymbol{y} = \boldsymbol{b}$ by $\tilde{{y}}_i$ and $y_i^\ast$, respectively.	
	If we denote the residual by $\tilde{\boldsymbol{r}} = \boldsymbol{b} - (z_j B -A) \tilde{\boldsymbol{y}}$, then the error $\tilde{\boldsymbol{y}} - \boldsymbol{y}^\ast$ satisfies
	\begin{equation}
	| \tilde{{y}}_i - {y}_i^\ast | \leq |\operatorname{Im} z_j|^{-1} \lambda_\mathrm{min} (B)^{-1} \| \tilde{\boldsymbol{r}} \|_2
	\label{eq:errorbound_linsys}
	\end{equation}
	for all $i = 1$, $2$, \dots, $n$, where $\lambda_{\min} (\cdot)$ is the smallest eigenvalue of a matrix and $\| \cdot \|_2$ denotes the Euclidean norm.
\end{theorem}

\begin{proof}
	Denote the square root of $B$ by $B^{1/2}$.
	Then, for all $i = 1$, $2$, \dots, $n$ we have
	\begin{align*}
	|\tilde{y}_i - y_i^* | & \leq \| \tilde{\boldsymbol{y}} - \boldsymbol{y}^\ast \|_2 \leq \| (z_j B - A)^{-1} \|_2 \| \tilde{\boldsymbol{r}} \|_2 = \| B^{-1/2} (z_j \mathrm{I} - B^{-1/2} A B^{-1/2})^{-1} B^{-1/2} \|_2 \| \tilde{\boldsymbol{r}} \|_2 \\ & \leq \| (z_j \mathrm{I} - B^{-1/2} A B^{-1/2})^{-1} \|_2 {\| B^{-1/2} \|_2}^2 \| \tilde{\boldsymbol{r}} \|_2 \leq |\operatorname{Im} z_j|^{-1} \lambda_\mathrm{min} (B)^{-1} \| \tilde{\boldsymbol{r}} \|_2.
	\end{align*}
	The bound $\| (z_j \mathrm{I} - B^{-1/2} A B^{-1/2})^{-1} \|_2 \leq (\operatorname{Im} z_j	)^{-1}$ can be geometrically interpreted as in Figure~\ref{fig:geo_bound}.
	Namely, the distance from the quadrature point $z_j$ to the nearest eigenvalue of $B^{-1/2} A B^{-1/2}$ is bounded below by the absolute value of the imaginary part of $z_j$.
	
	\begin{figure}[t]
		\centering
		\includegraphics[width=0.27\textwidth]{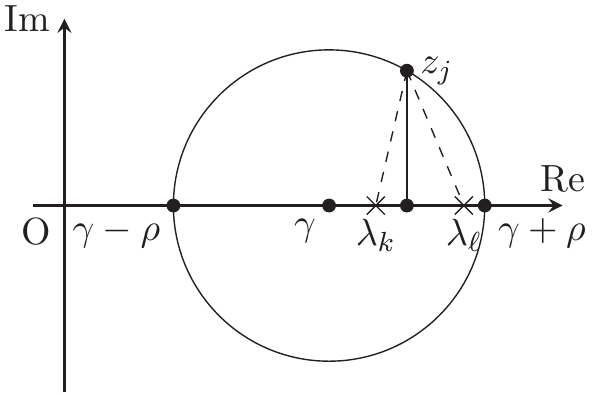}
		\caption{Geometric illustration for the bound $\| (z_j \mathrm{I} - B^{-1/2} A B^{-1/2})^{-1} \|_2 \leq (\operatorname{Im} z_j	)^{-1}$ in the complex plane.}
		\label{fig:geo_bound}
	\end{figure}
\end{proof}

Note that $z_j B - A$ is nonsingular for $j=1$, $2$, \dots, $N$, since $z_j$ is not in the real axis \eqref{eq:quadpoint}.
Hence, we do not need to verify the regularity of the coefficient matrix $z_j B - A$ such as in \cite{Rump2013JCAM}.
In addition, the bound \eqref{eq:errorbound_linsys} can be efficiently evaluated for sparse $A$ and $B$.
On the other hand, the bound~\eqref{eq:errorbound_linsys} shows that, if $\lambda_\mathrm{min}(B)$ is very small, the verification of $\tilde{\boldsymbol{y}}_j$ will be loose and the subsequent verification may fail.
This indicates that Theorem \ref{thm:efficientVNC} works well for well-conditioned $B$.
For ill-conditioned $B$, applying iterative refinements with multi-precision arithmetics \cite{OishiOgitaRump2009} to the linear system will potentially remedy the bound \eqref{eq:errorbound_linsys}.
%\begin{rem}
Furthermore, if each entry of $z_j B -A$ and $\boldsymbol{b}$ is not wide interval, one can use a \emph{staggered correction} \cite[Section 4.3]{Ogita2009}.
That is,
\[
| \tilde{{y}}_i - {y}_i^\ast | \le |\tilde{d}_i| + |\operatorname{Im} z_j|^{-1} \lambda_\mathrm{min} (B)^{-1} \|\boldsymbol{b} - (z_j B -A) (\tilde{\boldsymbol{y}}+\tilde{\boldsymbol{d}})\|_2,
\]
where $\tilde{\boldsymbol{d}}$ solves $(z_j B -A)\tilde{\boldsymbol{d}}\approx\tilde{\boldsymbol{r}}$ in a numerical (non-rigorous) sense and $\tilde{d}_i$ denotes the $i$\,th entry of $\tilde{\boldsymbol{d}}$.
This technique is expected to give sharper error bounds than \eqref{eq:errorbound_linsys} in Theorem \ref{thm:efficientVNC}.
%\end{rem}
\par
We summarize the above procedures in Algorithm~\ref{alg:method}.
In this implementation, we scale the target interval $\Omega$ into $[-1,1]$ by $A'=\frac{1}{\rho}\left(A-\gamma B\right)$ and compute the eigenvalues of $A'\boldsymbol{x}=\lambda'B\boldsymbol{x}$ for simplicity.
Here, we denote interval quantities with squares brackets.

The verification in line~4 of Algorithm~\ref{alg:method} can be done by, e.g., the following steps:
\begin{enumerate}
	\item Compute a numerical approximation $\tilde{\lambda}$ of $\hat{\lambda}$ (defined in Section \ref{sec:bound}) using MATLAB function~\texttt{eigs}.
	\item Set $c\in (0,1)$ such that $1<c|\tilde{\lambda}|$.
	\item Verify regularity of the interval matrix $[A]-[1,c|\tilde{\lambda}|]B$ by using INTLAB function \break \texttt{isregular}.
	%	\item Set $c$ such that $1 < c < |\hat{\lambda}|$ is tight
	\item Adopt $c|\tilde{\lambda}|$ as the lower bound of $| \hat{\lambda} |$.
\end{enumerate}

\begin{algorithm}
	\caption{Proposed method.}
	\label{alg:method}
	\begin{algorithmic}[1]
		\REQUIRE $A \in \mathbb{C}^{n \times n}$, $B \in \mathbb{C}^{n \times n}$, $L$, $M \in \mathbb{N}_+$ such that $m = L M$, $V \in \mathbb{C}^{n \times L}$, $\gamma$, $\rho \in \mathbb{R}$, and $\delta>0$.
		\ENSURE $[\lambda_i ]$, $i = 1$, $2$, $\dots$, $m$
		\STATE Scale $[A]=\left[\frac{1}{\rho}\left(A-\gamma B\right)\right]$.
		\STATE Set $N$ by \eqref{eq:setN}.
		\STATE Compute $[z_j] = [\mathrm{e}^{\mathrm{i} [\theta_j]}]$ with $[\theta_j] = [2 \pi / N (j - 1/2)]$ for $j = 1$, $2$, \dots, $N$.
		\STATE Rigorously compute a lower bound of $\left|\hat{\lambda}\right|=\min_{k=m+1,m+2, \dots,r}\left|\lambda_k\right|$.
		\STATE Compute $[|\mathsf{M}_{p, \mathrm{out}^{(N)}}|]$ with \eqref{eq:outer_mu} for $p = 0$, $1$, \dots, $2M-1$.
		\STATE Compute $[Y_j]$ for $j = 1$, $2$, \dots, $N$, by using \eqref{eq:errorbound_linsys} if $B$ is positive definite.
		\STATE Compute $[M_{p, \mathrm{in}}^{(N)}]$ by using \eqref{eq:rigorous_enclosure} for $j = 1$, $2$, \dots, $N$.
		\STATE Form $[H_M^{<, \mathrm{in}}]$ and $[H_M^\mathrm{in}]$.
		\STATE Rigorously compute the eigenvalues and the corresponding eigenvectors of the generalized Hankel eigenproblem $[ H_M^{<, \mathrm{in}} ] \boldsymbol{y} = \lambda' [H_M^\mathrm{in}] \boldsymbol{y}$.
		\STATE Rescale the eigenvalues $[\lambda_i] = [\rho \lambda'_i + \gamma]$ for $i=1$, $2$, $\dots$, $m$.
	\end{algorithmic}
\end{algorithm}

\section{Numerical examples}\label{sec:examples}
To illustrate effectiveness of the proposed method, we show three numerical examples (two artificially generated eigenproblems and one practical eigenproblem).
In first and third examples, we compared the proposed method with INTLAB's function~\texttt{verifyeig} in terms of the CPU time.
The second example was set for illustrating the performance of the proposed method under the case that the matrix $B$ is positive semidefinite or ill-conditioned.
All computations were carried out on Ubuntu 16.04, Intel(R) Xeon(R) Gold 6128 CPU @ 3.40~gigahertz (GHz) with 12 cores, 256~gigabytes (GB) random-access memory (RAM).
All programs were coded and run in MATLAB R2018a for double precision floating operation arithmetic with unit roundoff $2^{-53} \simeq 1.1 \cdot 10^{-16}$ and with INTLAB version~10.2~\cite{Rump1999}.
The matrix $V \in \mathbb{R}^{n \times L}$ was generated by using built-in MATLAB function~\texttt{randn}.
The tolerance of quadrature error was $\delta=10^{-15}$.
We determined the smallest $N$ that satisfies \eqref{eq:setN}.
Note again that the number of eigenvalues in the interval is given in advance.

In this example, numerically computed solutions of linear systems $(z_j B - A) Y_j = BV$ were obtained by using MATLAB function~\texttt{mldivide}.
The eigenvalues of $H_M^{<, \text{in}} \boldsymbol{y} = \lambda' H_M^{\text{in}}$  in line~9 of Algorithm~\ref{alg:method} were verified by using INTLAB function~\texttt{verifyeig}.

\paragraph{Artificially generated eigenproblems 1}
The test matrix pencil $z B - A$ used was given by 
\begin{align}
A = \mathrm{tridiag} (-1, 2, -1) \in \mathbb{R}^{n \times n}, \quad B = \diag (b_1, b_2, \dots, b_n) \in \mathbb{R}^{n \times n}, 
\label{eq:ABms}
\end{align}
where $\mathrm{tridiag}(\cdot, \cdot, \cdot)$ denotes the tridiagonal Toeplitz matrix consisting of a triplet and the value of $b_i$ normally distributes with mean $1$ and variance $10^{-7}$.
The generalized eigenproblem of matrix pencil \eqref{eq:ABms} models harmonic oscillators consisting of mass points and springs.
In particular, the matrix pencil \eqref{eq:ABms} arises from an equation of motion of mass points in one dimension.
Let $u_i (t)$ be the displacement of the $i$\,th point from the equilibrium of spring $i$ at time $t$ with mass $b_i$ and connected with two springs with stiffnesses $k_i = k_{i+1} = 1$.
Then, we have the equation $i = 1$, $2$, \dots, $n$
\begin{equation*}
b_i \frac{\mathrm{d}^2 u_i(t)}{\mathrm{d} t^2} = k_{i+1} (u_{i+1}(t) - u_i(t)) - k_i (u_i(t) - u_{i-1} (t)) = u_{i+1} - 2 u_i(t) + u_{i-1}(t).
\end{equation*}
Suppose that the mass point has a simple harmonic oscillation $u_i (t) = x_i \sin (w t + \phi)$, where $w$ is the angular rate, $\phi$ is the phase, and the homogeneous Dirichlet boundary condition $u_0 (t) = u_{n+1} (t) = 0$ is imposed.
Then, we have the eigenproblem $A \boldsymbol{x} = \omega^2 B \boldsymbol{x}$, where $\boldsymbol{x} = [-x_1, -x_2, \dots, -x_n]^\mathsf{T}$.
\par
The verification targets were four eigenvalues near $2$ for $n=2^\ell$, $\ell = 5$, $6$, \dots, $20$ of matrix pencil \eqref{eq:ABms}.
We set the parameters $L=2$ and $M=2$.
It is well-known that the eigenvalue of $A$ is given by $\lambda_i(A)=2-2\cos(i\pi/(n+1))$ for $i=1,2,\dots,n$.
Perturbation theory of Hermitian generalized eigenproblems \cite[Theorem~8.3]{Nakatsukasa:2011:APT:2395562} gives the following bound between $\lambda_i$ and $\lambda_i(A)$:
\begin{equation}\label{eq:perturbation}
|\lambda_i(A)-\lambda_i| \le |\lambda_i(A)| \|\Delta B\|_2 \|B^{-1}\|_2,
\end{equation}
where $\Delta B=I-B$.
Then, we derived the lower bound of $|\hat{\lambda}|$ using the eigenvalue $\lambda_i(A)$ with its bound \eqref{eq:perturbation}.

Figure~\ref{fig:fig1} shows the CPU time of the proposed method (Algorithm~\ref{alg:method}) and a standard method for verifying specific eigenvalues in MATLAB (build-in MATLAB function~\texttt{eigs} for the solution of the eigenproblem and INTLAB function~\texttt{verifyeig} for eigenvalue verification).
\begin{figure}[t]
	\centering
	\includegraphics[width=0.8\textwidth]{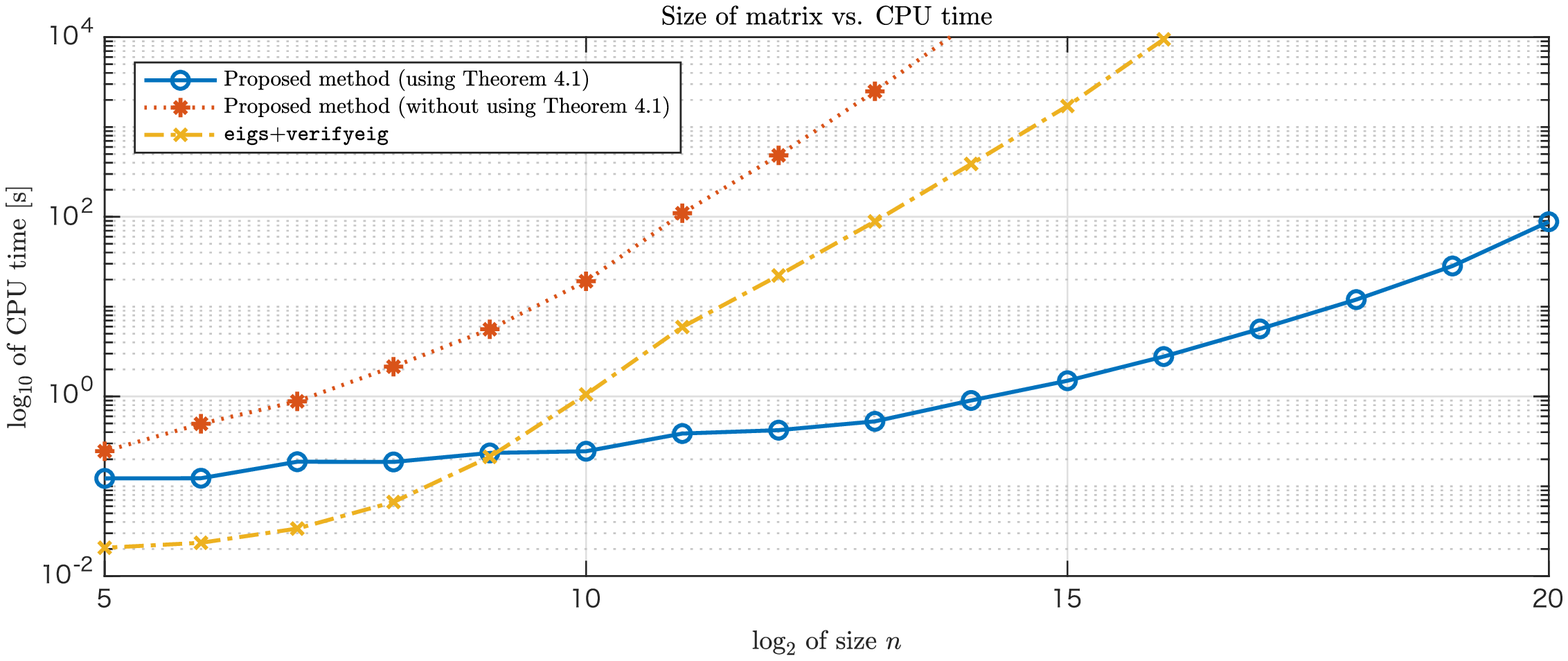}
	\caption{Comparison with \texttt{eigs}+\texttt{verifyeig} in terms of the CPU time.}
	\label{fig:fig1}
\end{figure}
As shown in Figure~\ref{fig:fig1}, the efficient verification technique based on Theorem~\ref{thm:efficientVNC} achieved a substantial improvement of the proposed method in the CPU time, and the proposed method using the technique based on Theorem~\ref{thm:efficientVNC} was faster than the standard method when the size of matrix $n$ is larger than $2^{10}$.
Furthermore, due to the limit of RAM, the standard method did not run for $\ell > 16$.
The proposed method tended to be more effective, as the size of the matrix $n$ becomes large and sparse.
On the other hand, the proposed method diminished more than \texttt{verifyeig} in terms of the error bounds.
Table~\ref{Tab:Ex1} gives the verified eigenvalues for the proposed method for each $\ell$.
For each $\ell$, the digits in single lines are the same as those of the exact eigenvalues, whereas the digits in double lines denote the supremum and infimum of the exact eigenvalues.
Table~\ref{Tab:Ex1} shows that the proposed method succeeded in verifying the eigenvalues at least 5 digits up to $\ell=20$.
For example, for $n=2^{10}$, \texttt{verifyeig} displayed correct 13 digits of the target eigenvalues
\begin{equation*}
1.9908051312881_{2}^{3},~
1.9969350358943_{5}^{6},~
2.0030649703896_{3}^{4},~
2.0091948771994_{3}^{4}.
%1.9908611656469_{8}^{9},~
%1.9969006214986_{7}^{8},~
%2.0030336470032_{1}^{2},~
%2.0092529816948_{7}^{8}.
\end{equation*}
This is mainly due to an overestimation of the error $\tilde{Y}_j - Y_j$ and in particular $\| (z_j B -A)^{-1} \|_2$ (see Theorem~\ref{thm:efficientVNC}).
In addition, we remark that this example~\eqref{eq:ABms} is very ideal to show the effectiveness of the proposed method, thanks to the sparsity of $A$ and $B$ and the simple structure of $B$.

\begin{table}[htbp]
	\caption{Verified eigenvalues for \emph{artificially generated problems 1}.}
	\centering
	%	\resizebox{\textwidth}{!}{
	\begin{tabular}{c|llll}
		\hline
		$\ell$ & \multicolumn{4}{c}{Eigenvalues near $2$}\\
		\hline
		5 & $1.715370325_{576319}^{647292}$, & $1.9048361618_{37662}^{83239}$, & $2.095163824_{458061}^{553889}$, & $2.284629679_{395375}^{449021}$ \\[1mm]
		6 & $1.8551304200_{49649}^{88382}$, & $1.95167256291_{1146}^{316}$, & $2.048327546_{241311}^{750658}$, & $2.144869623_{845262}^{957331}$ \\[1mm]
		7 & $1.9269559899_{20654}^{32997}$, & $1.97564721303_{2271}^{7844}$, & $2.02435287_{2964093}^{3629513}$, & $2.073044084_{849138}^{931074}$ \\[1mm]
		8 & $1.963329778_{797455}^{856981}$, & $1.9877759682_{78868}^{90228}$, & $2.0122240077_{57747}^{94568}$, & $2.03667024_{4513424}^{5464907}$ \\[1mm]
		9 & $1.981628386_{583106}^{639699}$, & $1.993876062_{151237}^{477787}$, & $2.006123967_{790464}^{896611}$, & $2.01837164_{2984799}^{475554}$ \\[1mm]
		10 & $1.99080513_{0338708}^{2328691}$, & $1.9969350_{32012484}^{40118075}$, & $2.003064970_{239924}^{538566}$, & $2.00919487_{492424}^{9374675}$ \\[1mm]
		11 & $1.99540031_{1711072}^{4525073}$, & $1.998466772_{568391}^{799085}$, & $2.0015332366_{65236}^{74622}$, & $2.004599697_{805267}^{937402}$ \\[1mm]
		12 & $1.997699590_{620666}^{95202}$, & $1.99923319_{1994498}^{2347755}$, & $2.000766798_{470323}^{533992}$, & $2.002300408_{831211}^{946773}$ \\[1mm]
		13 & $1.998849650_{636445}^{85915}$, & $1.999616549_{351151}^{578838}$, & $2.000383446_{084906}^{393767}$, & $2.0011503_{36316831}^{4576172}$ \\[1mm]
		14 & $1.9994247_{87629918}^{92481914}$, & $1.999808262_{201378}^{706644}$, & $2.000191734_{004361}^{485737}$, & $2.000575205_{226003}^{599759}$ \\[1mm]
		15 & $1.99971238_{5433325}^{6833255}$, & $1.99990412_{6923384}^{9815281}$, & $2.0000958_{6580425}^{74431981}$, & $2.00028761_{0865148}^{1877939}$ \\[1mm]
		16 & $1.999856_{177687935}^{205050025}$, & $1.9999520_{36785509}^{89032199}$, & $2.00004_{7548157389}^{8322762835}$, & $2.00014_{360912421}^{4007676757}$ \\[1mm]
		17 & $1.999928_{088198068}^{100686499}$, & $1.9999760_{2728946}^{34322243}$, & $2.00002396_{5890281}^{8809379}$, & $2.00007190_{3283288}^{4763985}$ \\[1mm]
		18 & $1.9999640_{4108165}^{5305932}$, & $1.9999880_{07405001}^{23539268}$, & $2.00001_{1955135737}^{2012521721}$, & $2.0000359_{28694293}^{75584551}$ \\[1mm]
		19 & $1.99998202_{1657987}^{5400559}$, & $1.99999_{399498807}^{4021890356}$, & $2.00000599_{0718144}^{3112617}$, & $2.00001797_{5380837}^{6886183}$ \\[1mm]
		20 & $1.9999910_{09237281}^{13961167}$, & $1.99999_{691993317}^{7090330345}$, & $2.00000_{2940857038}^{3052063462}$, & $2.0000089_{79158735}^{96723335}$ \\[1mm]
		\hline
	\end{tabular}%
	%	}
	\label{Tab:Ex1}
\end{table}

\paragraph{Artificially generated eigenproblems 2}
Another test matrix pencil $zB-A$ was considered for second numerical example, which is defined by
\[
A = \mathrm{pentadiag}(1,2,3,2,1)\in\mathbb{R}^{100\times 100},\quad B=\diag(1,1,\dots,1,b_{100})\in\mathbb{R}^{100\times 100},
\]
where ``$\mathrm{pentadiag}$'' denotes the pentadiagonal Toeplitz matrix. % and $b_{100}$ is the $(100,100)$-entry of $B$.
We changed $b_{100}$ as $0$,$10^{-16}$, $10^{-15}$, \dots, $10^{0}$ for illustrating the performance of our method under the case that $B$ is positive semidefinite or ill-conditioned.

We considered six ($m=6$) eigenvalues in $\Omega=[0.95,1.05]$.
We set the parameters $L=3$, $M=2$.
For the scaled eigenproblem, we verified $|\hat{\lambda}|>1.36$ by using INTLAB's function~\texttt{isregular}.

Figure~\ref{fig:fig5_2} shows a transition of verified partial eigenvalues with respect to $b_{100}$ entry.
The six target eigenvalues were plotted in Figure~\ref{fig:fig5_2}\,(a).
Changing $b_{100}$ entry, these values slightly moves between $b_{100}=1$ and $10^{-2}$.
Our proposed method succeeded in including these eigenvalues with the radius up to $10^{-9}$ as shown in Figure~\ref{fig:fig5_2}\,(b).
This result implies that our proposed algorithm works well in the case of the martix $B$ being semidefinite or ill-conditioned.
Finally, we remark that Theorem \ref{thm:efficientVNC} cannot work in this case because $\lambda_\mathrm{min}(B)^{-1}$ becomes very large or infinity.
One should use INTLAB's function~\texttt{verifylss} or another verification methods for linear systems.

\begin{figure}[tp]
	\centering
	\begin{minipage}{0.49\hsize}
		\centering
		\includegraphics[width=\textwidth]{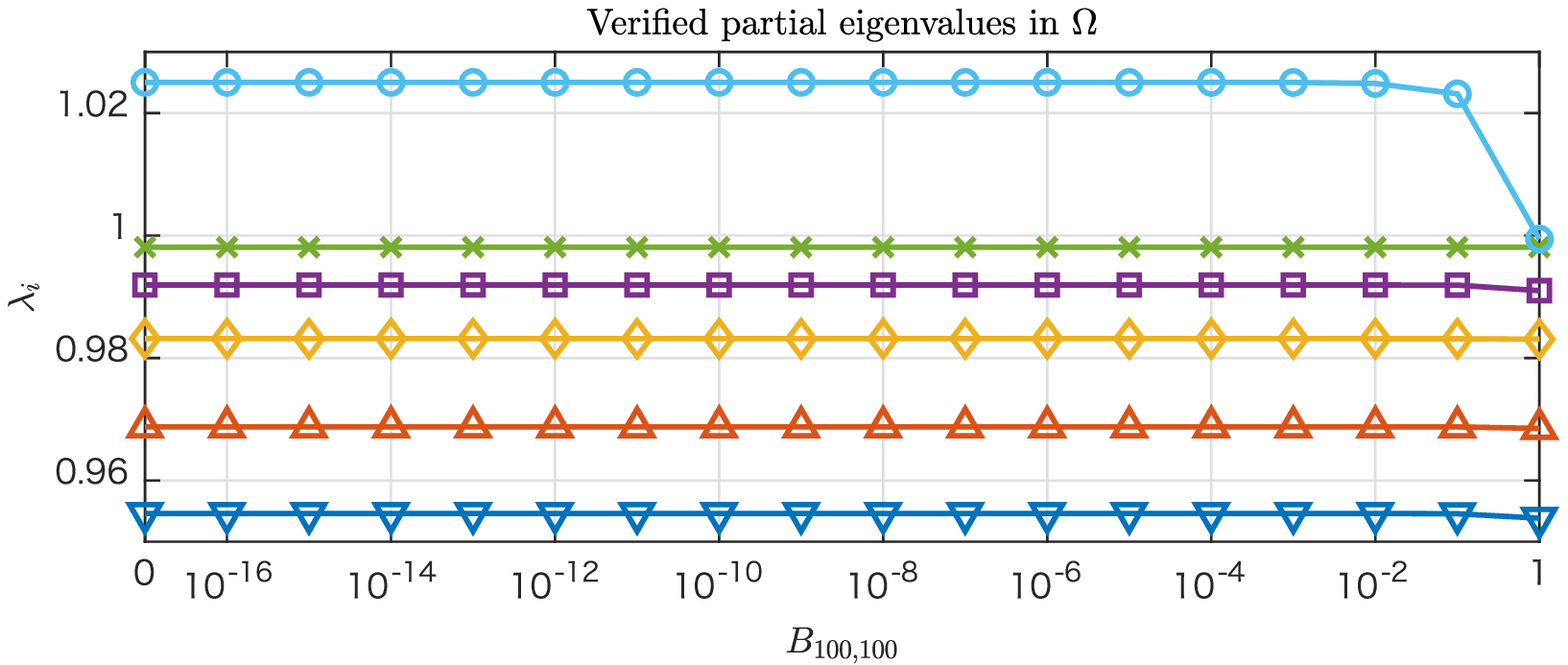}\\
		\subcaption{Transition of the six target eigenvalues}
		\label{fig:transision_semidef}
	\end{minipage}
	\begin{minipage}{0.49\hsize}
		\centering
		\includegraphics[width=\textwidth]{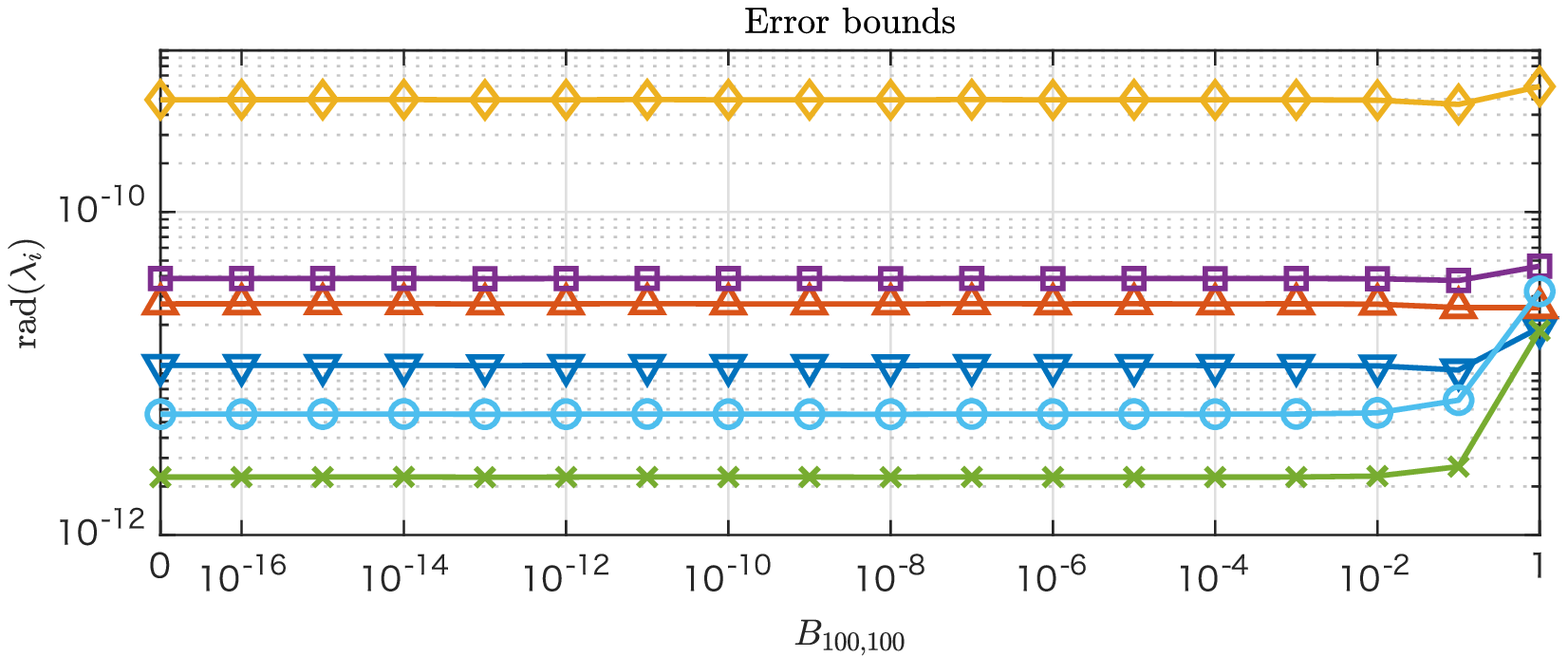}\\
		\subcaption{Radii of verified inclusions}
		\label{fig:results_semidef}
	\end{minipage}
	\caption{%
		Six target eigenvalues \subref{fig:transision_semidef} with the radii of verified inclusions \subref{fig:results_semidef}.
		Each symbol represents an eigenvalue with the same index.
	}
	\label{fig:fig5_2}
\end{figure}

\paragraph{Practical eigenproblems}
Finally, we consider a practical eigenproblem in quantum mechanics.
The verification targets are 52 eigenvalues in the interval $[-0.530, -0.425]$ of the Hermitian generalized eigenvalue problem for \emph{VCNT900}~\cite{CerdaSoria2000,ELSES,HoshiImachiKuwataKakudaFujitaMatsui2019}, which is associated with a vibrating carbon nanotube within a supercell with spd orbitals.
Both matrices $A$ and $B$ have nonzero density 42.8\% and are not sparse.
Figure~\ref{fig:fig2} shows the distribution of the 52 eigenvalues and the outer eigenvalues nearest to $[a, b]$.
To verify the lower bound of $\lambda_\mathrm{min} (B)$, we used Rump's method \cite{Rump2006BIT} using the INTLAB function \texttt{isspd}.
That is, we firstly computed an approximate smallest eigenvalue of $B$ (e.g., by built-in MATLAB function \texttt{eigs}), say $\tilde{\lambda}_\mathrm{min} (B)$.
We secondly checked the positive definiteness of $B-c\tilde{\lambda}_\mathrm{min} (B)I$ using \texttt{isspd} for a certain $c\in (0,1)$.
If the matrix is positive definite, then we adopt $c\tilde{\lambda}_\mathrm{min} (B)$ as the desired lower bound of $\lambda_\mathrm{min} (B)$.
Furthermore, for the scaled eigenproblem, we verified $|\hat{\lambda}|>1.19$ by using INTLAB's function~\texttt{isregular}.
Execution time of this part is about 120 seconds because of our naive implementation.
Indeed, there is a room to improve this part.
For example, we can use an efficient technique given in \cite{Yamamoto2001LAA}, which is based on Sylvester's law of inertia, to verify non-existence of the eigenvalues in the prescribed interval.

\par
The proposed method based on Theorem~\ref{thm:efficientVNC} successfully verified 37 of 52 eigenvalues in 7.9 seconds and failed to obtain the inclusion of the rest 15 eigenvalues.
This is due to an overestimation of the entries of $\tilde{Y}_j - Y_j$.
When using a verification method in INTALB (so-called backslash `\verb|\|') for the linear systems, the proposed method successfully verified all 52 eigenvalues in 36.0 seconds.
The standard method (\texttt{eigs}+\texttt{verifyeig}) also succeeded in verifying all 52 eigenvalues in 5.2 seconds, since the sizes of the matrices are not so large.
Although the most expensive part in Algorithm~\ref{alg:method} is the verification of $\tilde{Y}_j$, we note that  this can be done in parallel for all $j = 1$, $2$, \dots, $N$.

\begin{figure}[t]
	\centering
	\includegraphics[width=\textwidth]{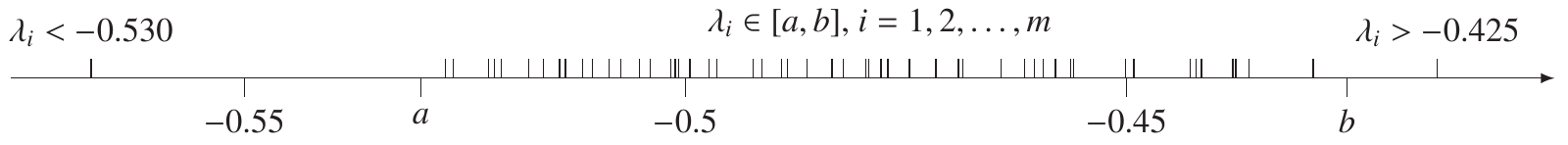}
	\caption{Eigenvalue distribution around $[a,b]=[-0.530, -0.425]$, which shows 52 inner eigenvalues and 2 outer ones. Ticks on the line denote each eigenvalue $\lambda_i$.}
	\label{fig:fig2}
\end{figure}

\section{Conclusions} \label{sec:conclusions}
We proposed a verified computation method for partial eigenvalues of a Hermitian generalized eigenproblem.
A contour integral-type eigensolver, the block Sakurai--Sugiura Hankel method, reduces a given eigenproblem into a generalized eigenproblem of block Hankel matrices consisting of complex moments.
The error of the complex moment can split into the error of numerical quadrature and the rounding error of numerical computations, which should be controlled rigorously.
We derived a truncation error bound of the quadrature and developed an efficient technique to verify the rounding error in the numerical solution of a linear system arising from each quadrature point.
Numerical experiments showed that, as the sizes of matrices become large and sparse, the proposed method outperforms a standard method on artificially generated eigenproblems.
It is also shown that proposed methods is applicable for practical eigenproblems.
We left the issue of how to verify the number of the eigenvalues in the prescribed interval.
Finally, we remark that the proposed method will be potentially efficient in parallel.
This is one of future directions for this research.

\section*{Acknowledgements}
We would like to thank Prof.~Yusaku Yamamoto for letting us know the work \cite{MiyataDuSogabeYamamotoZhang2009}.
We also would like to thank Prof.~Katsuhisa Ozaki for helpful discussions of parallel implementations.
This work was supported in part by the Faculty of Engineering, Information and Systems, University of Tsukuba.
The work of the first author was supported in part by JST/ACT-I (No.~JPMJPR16U6) and JSPS KAKENHI Grant Numbers~17K12690 and 18H03250.
The work of the second author was supported in part by JSPS KAKENHI Grant Number~16K17639 and Hattori Hokokai Foundation.
The work of the third author was supported in part by JSPS KAKENHI Grant Number~18K13453.

\appendix
\section{Regularity of a matrix pencil} \label{app:regular_pencil}
Consider verifying the regularity of matrix pencil $z B - A$ for Hermitian $A$ and Hermitian positive semidefinite $B$.
Recall that a matrix pencil $z B - A$ is said to be singular for square matrices $A$ and $B$ if $\det (z B - A)$ is identically equal to zero; regular otherwise.
The matrix pencil $z B - A$ is regular if and only if $\Ker (\left[ \begin{smallmatrix} A\\ B \end{smallmatrix} \right]) = \lbrace \boldsymbol{0} \rbrace$ \cite[Proposition 7.8.4]{Bernstein2018}.
Hence, we can guarantee the regularity of matrix pencil $z B - A$ by proving positive definiteness of $[B, A] [B, A]^\mathsf{H}$ by using the INTLAB function~\texttt{isspd} in \cite{Rump2006BIT}.

\bibliography{ref}

\end{document}